\documentclass[12pt,a4paper]{article}

\usepackage[utf8]{inputenc}
\usepackage{amssymb,amsmath}
\usepackage{graphicx}
\usepackage{color}
\usepackage{ulem}
\usepackage{hyperref}
\usepackage{extarrows}
\usepackage{floatflt}
\usepackage{caption}
\usepackage{mathtext}
\usepackage{amsfonts}
\usepackage{amsmath}
\usepackage{euscript}
\usepackage{amsthm}
\usepackage{framed}
\usepackage[justification=centering]{caption}

\usepackage{geometry}

\newlength{\wdth}

\newif\ifhide
\hidefalse

\titlepage

\newtheorem{Theorem}{Theorem}
\newtheorem{Lemma}[Theorem]{Lemma}
\newtheorem{remark}[Theorem]{Remark}
\newtheorem{definition}[Theorem]{Definition}

\newtheorem{Assumption}[Theorem]{Assumption}
\newtheorem{Corollary}[Theorem]{Corollary}
\newtheorem{Proposition}[Theorem]{Proposition}

\title{{\normalsize\tt\hfill\jobname.tex}\\
On consistency of Bayesian parameter estimations for a class of ergodic Markov models}

\author{A.I. Nurieva\footnote{National Research University Higher School of Economics, email: ai\_nurieva@mail.ru}, 
A.Yu. Veretennikov\footnote{Institute for Information 
Transmission Problems, email: ayv@iitp.ru}
}

\begin{document}

\maketitle
		
\begin{abstract}
The consistency of the Bayesian estimation of a parameter is shown for a class of ergodic discrete Markov chains. 
J.L. Doob's method was used, 
offered earlier 
for the i.i.d. situation. 
The result may be useful in the 
reliability theory for models with unknown parameters, 
in the risk management in financial mathematics, 
and in other applications. 

\end{abstract}


\smallskip



\section{Introduction}
Parameter estimation plays a significant and in some cases possibly even a crucial role in quite a few applications such as the reliability theory for models with unknown parameters, see \cite[chapter 3]{GBS}, in the Extreme Values theory for Markov processes, 
in the risk management in financial mathematics, et al. In the asymptotic sense, one of the basic desirable properties of any estimator in the long run is its consistency, weak or strong, as it shows that the estimation is close to the ``true'' parameter if the classical setting is accepted. Similarly, in the Bayesian setting consistency means literally the same -- convergence to the sample value of the parameter, even though there is no such thing as a ``true parameter value''  because it is to be sampled from the prior distribution. Also, as it is well-known, Bayesian estimators often work well in the classical setting, too, assuming some fictitious prior distribution for the parameter is chosen. 

In this paper the problem of strong consistency is  tackled for a certain class of Markov models in the Bayesian setting, and, as was already mentioned, in the classical situation with a fixed nonrandom ``true'' parameter value. Assume that there is a family of distributions $\{\mathbb P^\theta \}$ parameterised by some variable $\theta \in \Theta$, where $\Theta \subset R^m$ is a given parametric space. Any estimator is a measurable function of the observations, or, a bit more generally, a mapping from the space of outcomes $\Omega$, say, to the space $(\mathbb R^m, {\cal B}(\mathbb R^m))$ which is Borel measurable with respect to the sigma-algebra of the observations ${\cal F}^X$; here ${\cal B}(\mathbb R^m))$ is the Borel sigma-algebra in $\mathbb R^m$. 

In the Bayesian setting it is assumed that there is some prior distribution for $\theta$ on the set $\Theta$; the latter is usually a topological space, and in this paper, it will be assumed that $\Theta$ is a domain in $\mathbb R^m$ which is not necessarily bounded. S.N. Bernstein and R. von Mises were the first to establish consistency and the first steps towards the asymptotic normality of the Bayesian estimator for some particular i.i.d. cases, see \cite[Chapter IV, p.271]{Bern}, \cite[pp. 188-192]{vonMises}. The general theory about asymptotic normality was developed later by Le Cam \cite{LeCam} and Ibragimov and Khasmisnky \cite{IbrKhasm}; for more recent results see, for example, \cite{Kutoyants}, \cite{SpokoyniyPanov}. Another direction related to the problem was asymptotic singularity of measures for large observation samples based on martingale theory and developed in \cite{KLSh1, KLSh2, KLSh3, LPSh, Yashin}, et al. Naturally asymptotic normality requires more restrictive assumptions. On the other hand, ``just'' consistency may often be used for constructions of more efficient estimations by certain modifications. Also, in a situation where the conditions for asymptotic normality are not met, it may be even more desirable to know whether the applied estimator is consistent. Hence, it makes sense to separate the studies of sufficiency conditions for both properties, asymptotic normality and consistency. 

In this paper the approach offered for the i.i.d. observations in \cite{Doob} is used, adjusted for a class of markovian models. An important point in \cite{Doob} was a Strong Law of Large Numbers for the sample distribution functions (d.f. in what follows) as the number of observations tends to infinity. Also essential was an assumption that theoretical d.f. are different for different parameters. In this paper discrete densities on a  finite or countable state space are used. This restriction looks not crucial and likely may be  relaxed. At the level of ideas, the most close to this study is the paper \cite{Yashin}, where the earlier basic results from \cite{KLSh1, KLSh2, KLSh3} are applied precisely to the problem of parameter estimators' consistency. However, formally conditions in for this property \cite{Yashin} and in what follows are different. Also, in a way, this paper is based on a more simple background than that in \cite{KLSh1, KLSh2, KLSh3, Yashin}.

The paper consists of this Introduction, The setting, Auxiliary lemmata, Main result (theorem \ref{thm1}), and Proof of theorem \ref{thm1}.

\section{The setting}
Let$\{ X_t \}$ be a homogeneous Markov chain (MC) in discrete time $T=\{0,1,\ldots \}$ with a finite or countable (denumerable) state space $\mathcal{X} \subset {\mathbb R}^1$ (it will be clear in what follows why it is convenient to work on ${\mathbb R}^1$: although it is not a restriction, but it may be desirable that the elements of the state space are linearly ordered). The transition probabilities are denoted by $p_{ij}(s,t) = \mathbb P(X_t = j | X_s = i)) = \mathbb P(X_{t-s} = j | X_0 = i))=p_{ij}(t-s) $ for $s\le t$, and let  ${\cal P}(t) = (p_{ij} (t))$ be the transition probability matrix over time $t$; furthermore, they will all depend on a parameter $\theta$. The notion of ergodicity of a MC is not uniquely determined in the literature; in the present paper we understand it as follows. 

\begin{definition}\label{def1}
A homogeneous MC $(X_n, n  = 0,1,\ldots)$ is called ergodic if there exists a limiting invariant probability measure $\mu$ which does not depend on the initial distribution -- say, $\mu_0$ -- and to which there is a convergence in total variation for each $\mu_0$:
\begin{equation}\label{erg}
    \lim_{t\to\infty} \|p_{\mu_0,\cdot}(t) - \mu_\cdot\|_{TV} = 0,
\end{equation}
where $p_{\mu_0,j}(t) = \mathbb P_{\mu_0}(X_t=j)$. 
Recall that the total variation metric, or distance is given by the formula 
$$
\|\mu - \nu\|_{TV} := 2\sup_{A\in {\cal F}(\mathcal{X})}(\mu(A)-\nu(A)).
$$
\end{definition}
As it was said, the transition probabilities depend on a parameter and the problem under consideration is estimation of this parameter given observations on the time interval $[1,n]$ where $n\to\infty$. It is assumed that $\theta \in \Theta \subset \mathbb{R}^m$; $\Theta$ is a  domain, not necessarily bounded. Naturally, a stationary measure, generally speaking, also depends on $\theta$: denote it from now on by $\mu^\theta(dx)$ and note that under the assumption of convergence (\ref{erg}) it is necessarily unique. We will need the extended process $Y_n = (X_n,X_{n+1})$ which is also a MC on the state space $\mathcal{X}\times \mathcal{X}$.
The symbol $\mu^\theta(dx, dx')$ 
will denote the stationary measure for the MC $(Y_n)$; it is easy to see that such an invariant measure does exist. Assume that the functions $p^\theta(\cdot,\cdot)$ are Borel measurable with respect to the variable $\theta$. Then due to the ergodicity (see (\ref{erg})) the invariant probabilities are also Borel measurable in $\theta$. Following Doob's approach, {\it suppose} that a (weak) Law of Large Numbers (LLN) holds true for the MC $(X_n)$ with respect to the corresponding measure $\mathbb P^\theta$, for each  $\theta$ . In this case, LLN is also valid for the MC $(Y_n)$, where $Y_n:=(X_n, X_{n+1})$. It is easy to see that these two conditions -- LLN for the MC $(X_n)$ and for the MC $(Y_n)$ -- are equivalent. Hence, the following assumption will be accepted in what follows. 
\begin{Assumption}\label{def2}
It is assumed that for each $\theta\in \Theta$ and any measurable $A,B$ a convergence holds true, 
$$
 \left|\frac1{T} \sum_{s=0}^{T-1} 1(X_s\in A, X_{s+1}\in B) -  \mu^\theta(A\times B)\right| \stackrel{\mathbb P^\theta}\to 0, \quad T\to\infty.
$$
\end{Assumption}
\noindent
This assumption is equivalent also to the condition 
$$
 \left|\frac1{T} \sum_{s=0}^{T-1} g(Y_s) -  \int g(y)\mu^\theta(dy)\right| \stackrel{\mathbb P^\theta}\to 0, \quad T\to\infty,
$$
for any bounded measurable function $g(y)$, where  $y=(x,x')$. 

 \medskip

Let us collect the comments made earlier in the form of a proposition. 
\begin{Proposition} \label{claim1}
Under the assumptions made above the following statements hold: 
\begin{enumerate}
\item If $(X_n)$ is a homogeneous MC then $Y_n$ is also a homogeneous MC. 
\item If the MC $(X_n)$ is ergodic then the MC $(Y_n)$ is also ergodic, and vice versa. 

\end{enumerate}
\end{Proposition}
Note that, as usual, all sigma-algebras in the text are regarded as completed with respect to the corresponding probability measures.

\section{Main result}
The Bayesian setting assumes that the parameter $\theta$ is random; let it have a prior probability  distribution $\mathbb Q$ on $\Theta$. 
Recall that here $\Theta$ is a domain in ${\mathbb R}^m$, not necessarily bounded. 
It is assumed that 
\begin{equation}\label{eq3-1a}
\mathbb E\theta < \infty. 
\end{equation}
Any estimator of the parameter given observations is represented  by some Borel measurable function $\hat\theta_n = \hat\theta_n (X_1, \ldots, X_n)$. As it is well-known (cf., for example, \cite[chapter 19]{Borovkov1}), there exists a Borel measurable function $\phi_n$ such that the Bayesian estimator reads, 
$$
\mathbb E(\theta |X_1, \ldots, X_n) \stackrel{\text{$\mathbb P^\theta$-a.s.}}= \phi_n(X_1, \ldots, X_n). 
$$
So, the statistic $\hat \theta_n:=\phi_n(X_1,\dots, X_n) = \mathbb E(\theta|X_1,\dots, X_n)$ is necessarily $({\cal F}^X_N, {\cal B}({\mathbb R}^m))$-measurable; hence, also $({\cal F}^X_\infty, {\cal B}({\mathbb R}^m))$-measurable, and $\phi_n$ is measurable with respect to the pair of $\sigma$-algebras $({\cal B}({\cal X})^{n}, {\cal B}({\mathbb R}^m)), \forall n \in \mathbb N$, 
where ${\cal B}({\cal X})$ is the set of all subsets of the state space $\cal X$, that is, ${\cal B}({\cal X})= 2^{\cal X}$. Recall that a pointwise limit of measurable functions is also measurable.

\begin{Theorem}\label{thm1}
Let the following conditions be satisfied:
\begin{enumerate}
    
\item  
Transition probability matrices of the MC $(X_n)$ for different values of $\theta$ are different,
that is, for any $\theta\not = \theta'$ there exist $i,j$ such that $p^\theta_{ij} \not = p^{\theta'}_{ij}$.

\item 
Let MC $(Y_n)$ be ergodic for each $\theta$  under the measure $\mathbb P^\theta$ in the sense of the definition (\ref{def1}), and let the (weak) LLN hold for the process $Y$ for each $\theta$  in the sense of the assumption (\ref{def2}). 
Then there is a 
convergence 
\begin{equation} \label{formula3}
\hat\theta_n 
\stackrel{\mbox{}}{\to} \theta, \quad n\to\infty, \quad \mathbb P-\text{a.s.}
\end{equation}
\end{enumerate}

\end{Theorem}
Here, as usual in the Bayesian setting, 
$$
\mathbb P(d\theta, d\omega) = \mathbb Q(d\theta) \mathbb P^\theta(d\omega).
$$
\noindent
\begin{remark}
Recall that similar results under different conditions were established in \cite[theorems 1-2]{Yashin}. Formally, those conditions in \cite{Yashin} may be applicable, or not  applicable in our situation because the assumption of the absolute continuity for the projection measures on the sigma-algebra ${\cal F}^X_n$ for any two values of the parameter is not assumed, see   \cite[Theorem 1, condition (C)]{Yashin} and  \cite[Theorem 2, condition (b)]{Yashin}. In markovian examples in \cite[\S13]{KLSh2} a similar condition to \cite[Theorem 1, condition (C)]{Yashin} was  assumed as well, see  theorem 22, condition (b). 
In the present paper such a condition is neither assumed, nor it follows from the other assumptions. Intuitively, the lack of continuity should only help consistency; nevertheless, even if so, it apparently does require some calculus. In any case, the proof of the theorem \ref{thm1} in what follows does not distinguish between the cases tackled in \cite{Yashin} and the cases not covered by this cited paper.
\end{remark}

\begin{remark}
As in the setting of Doob in \cite{Doob}, this result may also be  used in the classical setting where $\theta$ is not random and there exists a unique ``true'' parameter value. For that, an artificial prior density should be introduced on $\Theta$ which must be everywhere positive. Then, as in \cite{Doob}, the analogous assertion will hold true about an almost sure convergence of the artificial Bayesian estimator under the product  measure on $\Theta \times {\mathcal X}^\infty$. 

In particular, what is usually highlighted about Bernstein and von Mises theorem is that if the measure $\mathbb Q$ has a densty $q(\theta)$ which is everywhere positive, then convergence of the Bayesian estimator towards $\theta$ will take place almost everywhere in $\Theta$ with respect to the Lebesgue measure. Actually, it suffices for this property that the measure $\mathbb Q$ were absolutely continuous with respect to the latter.  
However, in either case there is no way to know for which particular values of $\theta$ this convergence is valid and for which maybe not; it may only be claimed that the set of ``bad'' values of $\theta$ with no convergence has measure zero.

\end{remark}

\section{Auxiliary results}
Let us define the sample distribution function 
$$ 
F_N (x, x'):=
\frac1{N}\sum_{t=0}^{N-1} 1(X_t \le x, X_{t+1} \le x'). 
$$ 
Denote by ${\cal S}=\{F(x,x'), x,x'\in \mathbb R\}$ the space of all functions of two variables $(x,x')$ with the following properties: 

\begin{enumerate}
 \item
 $0\le F(x,x')\le 1$ for each $x,x'\in \mathbb R$. 
 \item
If $x\le z,\, x'\le z'$, then $F(x,x')\le F(z,z')$ (monotonicity). 
\item For each $x,x'\in R$
\[
\lim_{z\downarrow x, z'\downarrow x'} F(z,z') = F(x,x'). 
\]

\item
For each $x,x'\in \mathbb R$ there exists a limit 
\[
\lim_{z\uparrow x, z'\uparrow x'} F(z,z') =: F(x,x')_-. 
\]
(NB: Actually, the latter notation will not be used in what follows; it is just an analogue of the one-dimensional property of ``l\`ag'' -- possessing  ``limites \`a gauche'' -- for the one-dimensional case. Respectively, the property 3 is the analogue of the ``c\`ad'' -- being  ``continue \`a droite'' for a function of one variable.)
 \item
\[
\lim_{z\uparrow +\infty, z'\uparrow  +\infty} F(z,z') = 1.
\] 
\item
\[
\lim_{z\downarrow -\infty, z'\downarrow  -\infty} F(z,z') = 0.
\] 
\end{enumerate}

In fact, in the situation under the consideration we deal with some proper subset of all distribution functions of two variables, because all  corresponding measures on $\mathbb R^2$ have atoms in our setting. However, all we need is that this more general space of distribution functions  with a certain metric is a Polish space, and this will be guaranteed by proposition \ref{pro9} in what follows. 

\medskip

Denote by $\Sigma({\cal S})$ the sigma-algebra on ${\cal S}$ generated by all finite cylinders, i.e., 
$$
\Sigma({\cal S}) := \sigma(F \in {\cal S}: \, F(x_1, x_1')\le a_1, \ldots F(x_n, x_n')\le a_n))
$$
for any  $(x_1,x'_1), \ldots, (x_n,x'_n) \in \mathbb R^2$ and $a_1,\ldots, a_n \in [0,1]$. 

\smallskip

Note that the distribution function of any two-dimensional random vector belongs to the space ${\cal S}$, and all sample d.f. $F_N$ belong to this space, too.  

\begin{Lemma} \label{lem5}
The function $F_N: \, \Omega \mapsto \mathbb R$ is a measurable map with respect to the corresponding pair of sigma-algebras $({\cal F}_N^X; \Sigma(S))$.

\end{Lemma}

\begin{proof}
The proof is elementary and is shown here only for the convenience of the reader. Indeed, for any couple  $(x,x')$ the mapping 
$$ 
F_N (x, x'):=
\frac1{N}\sum_{t=0}^{N-1} 1(X_t \le x, X_{t+1} \le x')
$$ 
is measurable as a function of $\omega$, as a finite sum of random variables (indicators), which may be expressed by the relation
$$
\displaystyle (\omega:\; F_N (x, x')\le a) \in {\cal F}^X_N
$$
for any $a\in R$. Then,  for any finite sets of $(x_1,x'_1), \ldots, (x_n,x'_n)$ and $a_1,\ldots, a_n$ we have, 
$$
\displaystyle (\omega:\; F_N (x_1, x_1')\le a_1, \ldots F_N (x_n, x_n')\le a_n)) \in {\cal F}^X_N, 
$$ 
by the definition of what is a sigma-algebra. Therefore,  $F_N (\cdot, \cdot)$ as a function of  $\omega$ is, indeed, $({\cal F}_N^X; \Sigma(S))$-measurable, as required. 
\end{proof}

\medskip

In the next lemma it is assumed that the distribution of $Y_0 = (X_0, X_1)$ is invariant. In this case its distribution function is denoted by $\hat F^\theta (x, x')$; recall that due to ergodicity it is unique. It may be presented by the formula
\begin{equation}\label{invd2}
\hat F^\theta(x,x') = \hat F^\theta(x)p^\theta_{x,x'}, 
\end{equation}
where, in turn, $\hat F^\theta(x)$ is the (unique) invariant distribution function of the MC $X_n$ with respect to the probability measure $P^\theta$,  which is simultaneously the limiting distribution function for the $(X_n)$.

\begin{Lemma} \label{lem6}
Under the assumption that all transition probabilities $p_{ij}^\theta,\, i,j\in {\cal X}$ are Borel measurable in $\theta$, the invariant distribution function $\hat F^\theta (x, x')$ is Borel measurable in $\theta$ for each pair $(x, x')$. 
\end{Lemma}
\begin{proof}
Indeed, invariant probabilities $p_{inv}^\theta(i),\, i\in {\cal X}$ are measurable in $\theta$ as limits of measurable $n$-step transition probabilities. So, the ``double'' invariant probabilities $p_{inv}^\theta(i)p_{}^\theta(ij),\, i,j\in {\cal X}$ also have the same property. Hence, the theoretical d.f. 
$$
\mathbb P^\theta_{x_0}(X_n\le x, X_{n+1} \le x') 
= \sum_{i\le x} p_{x_0,i}^\theta(n) \sum_{j\le x'}p^\theta_{i,j}
$$
is clearly Borel measurable in $\theta$, too. So is its limit at $n\to\infty$ which equals $\hat F^\theta (x, x')$, as required.
\end{proof}

\medskip

\noindent
Let us recall the L\`evy--Doob theorem on convergence of conditional expectations.
\begin{Proposition} \label{Levy}
(see, e.g., \cite[Theorem 4.3.10]{Krylov})
Let $E |\xi| < \infty$ and let   ${{\cal F}_n}, n = 0,1, \dots$ be an increasing sequence of $\sigma$-algebras, ${\cal F}_n \subset {\cal F}_{n+1}$, and let ${\cal F}_{\infty} $ be the minimal $\sigma$-algebra which contans all ${\cal F}_n$, that is, ${\cal F}_{\infty} = \bigvee\limits_{n} {\cal F}_n$ (that is the minimal sigma-algebra generated by all ${\cal F}_n$). Then
$$
\lim_{n \to \infty} \mathbb E(\xi|{\cal F}_n) = \mathbb E(\xi|{\cal F}_{\infty}) , \quad \mbox{a.s.}
$$ 
and 
$$
\lim_{n \to \infty} \mathbb E| \mathbb E(\xi|{\cal F}_{\infty}) - \mathbb E(\xi|{\cal F}_n)| = 0.
$$
\end{Proposition}
In our setting due to the proposition \ref{Levy} we have, 
$$
\lim_{n \to \infty} \mathbb E(\theta|X_1,\dots, X_n) = \lim_{n \to \infty} \phi_n(X) = \mathbb E(\theta |{\cal F}^X_{\infty}) \quad \mbox{a.s.}
$$
This implies that the limit in the left hand side in the latter  double equality  is ${\cal F}^X_{\infty}$-measurable. 

\smallskip

\begin{Lemma}\label{lem7}
Assume that the transition probability matrices are different for different parameter values, that is, $\theta\not = \theta'$ implies that there exist $i,j$ such that $p^\theta_{ij} \not = p^\theta_{ij}$. Then the mapping $\theta \mapsto \hat F^\theta(j,j'), j,j'\in \mathcal{X}$ is one-to-one. Moreover, the mapping 
\begin{equation}\label{G}
G:\; \theta \mapsto \hat F^\theta (x, x'), \quad x,x'\in \mathbb R,
\end{equation}
is also one-to-one.

\end{Lemma}

\begin{proof}
The proof follows from the formula  (\ref{invd2}). Indeed, if for $\theta \neq \theta'$ the one-dimensional invariant distribution functions $\hat F_\theta(\cdot)$ are different, then two-dimensional are different, too. If for some pair  $\theta \not = \theta'$ the one-dimensional d.f. coincide, $\hat F^\theta(\cdot) = \hat F^{\theta'}(\cdot)$, then the two-dimensional one are yet different due to the formula  (\ref{invd2}) and by virtue of the distinguishability assumption of transition probabilities for different parameter values. The same property for the mapping $G$ follows straightforwardly.
\end{proof}

Further, due to the assumed LLN the following convergence of relative frequencies holds, 
$$
\frac1{n}\sum_{t=0}^{n-1} 1(X_t \le j) \stackrel{ \mathbb P^\theta}\to \hat F^\theta(j) = \mathbb E_{inv}^\theta 1(X_0 \le j), \quad n\to\infty, 
$$ 
where $ \mathbb E_{inv}^\theta$ is expectation with respect to the corresponding invariant measure. A similar convergence holds true for two-dimensional relative frequencies, 
$$
\frac1{n}\sum_{t=0}^{n-1} 1(X_t \le j, X_{t+1} \le j') \stackrel{P^\theta}\to \hat F^\theta(j,j') = \mathbb E_{inv}^\theta 1(X_0 \le j, X_1 \le j'), \quad n\to\infty. 
$$
Since two-dimensional invariant d.f. $\hat F^\theta(\cdot,\cdot)$ are different for any two different parameter values, the value $\theta$ is uniquely determined by the infinite trajectory of observations $X = (X_n, n=1, \ldots )$. In other words, the mapping $\theta \mapsto \hat F^\theta(\cdot,\cdot)$ is one-to-one. This mapping is measurable due to the LLN and because the limit of measurable mappings is also measurable. Moreover, as it follows from proposition \ref{pro8} (see below; it is not linked to this lemma), the inverse mapping is also measurable. 

Let us recall some further definitions; it is necessary because one of them is not standard in most of mathematics areas (see definition \ref{def9} in what follows).

\begin{definition}
Borel measurable sets in a Polish (\& more generally, in any topological) space $\mathbb{X}$ are the sets of the minimal $\sigma$-algebra $\cal{B}(\mathbb{X})$ of subsets in $ \mathbb{X}$ which contains all open subsets in  $ \mathbb{X}$.
\end{definition}

\begin{definition}\label{def9}
Let $X, Y$ be Borel measurable sets in Polish spaces ${\mathbb X}, {\mathbb Y}$, respectively. The mapping $f : X \to Y$ is called:
\begin{enumerate}
\item Borel iff its graph $\Gamma_f = \{(x,y): x\in X, f(x) = y\}$ is a Borel set in the space $\mathbb{X} \times \mathbb{Y}$;

\item $B$-measurable iff the image of any Borel set from the space $Y$ under the inverse mapping $f^{-1}$ is a Borel set in  ${\mathbb X}$. 
\end{enumerate}
\end{definition}
Note that the ``usual'' definition of a Borel function in the majority of areas of mathematics coincides with 10.2.  

The next result may be found in \cite[Theorem 2.4.3]{KanoveyLyubetskiy} (we only state the part of this theorem which will be used in what follows). 
\begin{Proposition}[{\color{black}\cite[Theorem 2.4.3]{KanoveyLyubetskiy}}]\label{pro8}
Let $X,Y$ be Borel sets in Polish spaces and $f: X \to Y$ be some mapping. Then:
\begin{enumerate}
\item If $f$ is Borel measurable then the images of all Borel sets from $Y$ under the inverse mapping $f^{-1}$ are also Borel, so that the mapping $f$ is $B$-measurable;

\item Vice versa, if $f$ is $B$-measurable then it is a Borel function. 
    
\end{enumerate}
\end{Proposition}

\begin{Corollary}\label{cor11}
If the mapping $f$ is Borel measurable and one-to-one, then its inverse $f^{-1}$ is $B$-measurable and, hence, Borel one in the sense of  definition \ref{def9}.
\end{Corollary}

In order to apply proposition \ref{pro8} in the proof of our main result in the next section, let us show that both proposition and its corollary \ref{cor11} are applicable to the mapping $G$ (see (\ref{G})). 

\begin{Lemma}\label{lem4-4}
Under the assumptions of theorem \ref{thm1} the mapping $G^{-1}$ is Borel and B-measurable.
\end{Lemma}
\begin{proof}
Firstly, the mapping $G: \theta \mapsto \hat F^\theta (\cdot, \cdot)$ is $B$-measurable in the sense of the definition \ref{def9}. Indeed, the element $\hat F^\theta (\cdot, \cdot)$ is a limit in probability $\mathbb P^\theta$ of the sequence of functions $\mathbb E^\theta F_N (\cdot, \cdot)$, which are all $B$-measurable in $\theta$; therefore, so is their limit. 

Secondly, according to lemma \ref{lem7}, the mapping $G$ is one-to one; hence, so is its inverse is $G^{-1}$. The claim of lemma  \ref{lem4-4} now follows from corollary \ref{cor11}.
\end{proof}

\medskip

Further, it is desirable that the parametric space $\Theta$ and the space of invariant distribution functions were complete and separable metric spaces. It is trivial with $\Theta \subset \mathbb{R}^m$ with the Euclidean metric; for the space of ``double'' distribution functions a suitable matric should be chosen which is, of course, not unique. To each distribution function there correspond a probability distribution on ${\mathbb R}^{2}$. Let us accept that the distance between two distribution functions is defined as a distance between their corresponding measures. Let us choose Prokhorov's metric $d_p(\nu_1, \nu_2)$ for them: if  $\alpha$-neighbourhood of a set $A \subset {\mathbb R}^{2}$ is denoted by 
$$
A_\alpha \!:= \!\{\bar a\!:=\! (a_1,a_2)\in {\mathbb R}^{2}: \; d(\bar a,A) \!<\! \alpha\}, \;  \text{if} \; A \neq \emptyset, \;  \emptyset_\alpha = \emptyset \quad \forall \alpha >0,
$$
then the distance between probability measures $\nu_1, \nu_2$ on ${\mathbb R}^{2}$ is defined by the formula 
$$
d_p(\nu_1, \nu_2)\!:=\!\inf \{ \alpha \!>0: \nu_2(A) \!\leq\! \nu_1(A_\alpha) \!+\! \alpha \; \text{\&} \;  \nu_1(A) \!\leq\! \nu_2(A_\alpha) \!+\! \alpha , \; \forall A \in \mathcal{B} (S) \}.
$$
The same formula provides the distance between two distribution functions, namely, as a distance $d_p(\cdot, \cdot)$ between the corresponding measures on 
${\mathbb R}^{2}$.

\begin{Proposition}[{\color{black}\cite[Lemma 1.4]{Prokhorov}}]\label{pro9}
Let a metric space be complete and separable. Then the space of probability measures on it with the Prokhorov metric is also complete and separable. 
\end{Proposition}

\section{Proof of theorem 
\ref{thm1}}

\begin{proof} 
By virtue of L\`evy--Doob's theorem (see proposition \ref{Levy}) we have, 
\begin{equation}\label{thetaB}
\hat\theta_n \equiv \mathbb E(\theta|X_1,\ldots,X_n) = \mathbb E(\theta |{\cal F}^X_n) \to \mathbb E(\theta |{\cal F}^X_\infty) =: \hat\theta_\infty, \quad n\to\infty \quad \mbox{$P$-a.s.}.
\end{equation}
Due to its definition, the random variable $\hat\theta_\infty$ is  ${\cal F}^X_\infty$-measurable; being a conditional expectation, it is a Bayesian estimator of $\theta$ constructed upon the infinite sequence of observations  $X_1, X_2, \ldots$ For the proof of the theorem, it suffices to establish the equality 
\begin{equation}\label{thetatheta}
\hat\theta_\infty \stackrel{\mathbb P-\text{a.s.}}= \theta.
\end{equation}
The basis for thsi equality is the empirical fact that $\theta$ is uniquely deterined by the infinite sequence of observations due to the assumed LLN and because of the one-to-one correspondence between  $\theta$ and the invariant distribution of the pair $(X_0,X_1)$. Let us provide more rigorous considerations related, in particular, to the measurability. 

By virtue of the LLN assumptions, we have
$$
{\cal F}^X_{N-1} \ni F_N (x) = \frac{1}{N} \sum_{i=0}^{N-1} I(X_i \leq x) \stackrel{\mathbb P^\theta}\to  \hat F^\theta (x)=\mathbb E_{inv}^\theta I(X_0 \leq x),
$$
and also
$$
{\cal F}^X_{N} \!\ni \! F_N (x, x') \!=\! \frac{1}{N} \sum_{i=0}^{N-1} I(X_i \leq x, X_{i+1} \leq x') \!\stackrel{\mathbb P^\theta}\to\!  \hat F^\theta (x, x')\!=\!\mathbb E_{inv}^\theta I(X_0 \!\leq \!x, X_1 \!\leq \!x').
$$
The random variable $F_N (x, x')$ is $({\cal F}^X_{N}, {\cal B}({\mathbb R}^{2}))$-measurable for any pair $(x, x') \in {\mathbb R}^2$. 
According to lemma \ref{lem5}, the mapping $F_N (\cdot, \cdot)$ is $({\cal F}^X_N, \Sigma({\cal S}))$-measurable, hence, it is a random variable in the space of distribution functions. 

Now, according to lemma \ref{lem7} the following equality holds true, 
$$
\theta \stackrel{\mathbb P-\text{a.s.}}=G^{-1}(\,\underbrace{\hat F^\theta (\cdot, \cdot)}_{\in {\cal F}^X_\infty}\,). 
$$
By virtue of lemma \ref{lem4-4}, the mapping $G^{-1}$ is Borel and B-measurable. Therefore, 

$$
G^{-1}(\hat F^\theta (\cdot, \cdot)) \in  {\cal F}^X_\infty.
$$
Therefore,  by virtue of (\ref{thetaB}), 
\[
\hat\theta_n \to \mathbb E(\theta |{\cal F}^X_\infty) 
\stackrel{\mathbb P-\text{a.s.}}=\mathbb E(G^{-1}(\hat F^\theta (\cdot, \cdot))|{\cal F}^X_\infty) \stackrel{\mathbb P-\text{a.s.}} = G^{-1}(\hat F^\theta (\cdot, \cdot))
\stackrel{\mathbb P-\text{a.s.}} = \theta.
\]
This means that (\ref{thetatheta}) holds true.  implies the desired convergence 
(\ref{formula3}). 
Theorem \ref{thm1} is proved. 
\end{proof}

\section*{Acknowledgements}
For both authors this study is supported by the Russian Foundation for  Basic Research, grant 20-01-00575a. Theorem \ref{thm1} and lemma \ref{lem6} are established by the first author; lemmata \ref{lem5}  and \ref{lem7} are proved by the second author.
%


\end{document}